\RequirePackage{fix-cm}
\documentclass{svjour3}                     
\smartqed  
\usepackage{epsfig}
\usepackage{amsmath}
\usepackage{amssymb}
\usepackage{amscd}
\usepackage{graphicx}
\usepackage{verbatim}

\numberwithin{equation}{section}

\newcommand{\na}{\mathbb N}

\newcommand{\Z}{\mathbb Z}
\newcommand{\R}{\mathbb R}

\newcommand{\sq}{sequence}
\newcommand{\xt}{$(X,T)$}

\newcommand{\tl}{topological}
\newcommand{\im }{invariant measure}
\newcommand{\inv}{invariant}
\newcommand{\ds}{dynamical system}

\newcommand{\htop}{\mathsf h_{\mathsf{top}}}
\newcommand{\sep}{\mathsf{sep}}
\newcommand{\graph}{\mathsf{graph}}

\newcommand{\mL}{L\kern-0.08cm\char39}

\begin{document}
	
	\title{Minimal spaces with cyclic group of homeomorphisms
		   \thanks{The first and third authors were supported by the NCN (National Science Center, Poland) grant 2013/08/A/ST1/00275. The second author was supported by VEGA grant  1/0786/15.}
		   }
			

	\author{Tomasz Downarowicz   \and
		\mL ubom\'\i r Snoha \and Dariusz Tywoniuk
	}

	
	\institute{Tomasz Downarowicz \at
		Institute of Mathematics of the Polish Academy of Science, \'Sniadeckich 8, 00-956 Warszawa, Poland \\
		\email{downar@pwr.wroc.pl}           
	\and
	\mL ubom\'\i r Snoha \at
	Department of Mathematics, Faculty of Natural Sciences, Matej Bel University, Tajovsk\'eho 40, 974 01 Bansk\'a Bystrica, Slovakia\\ 	\email{Lubomir.Snoha@umb.sk}           
	\and
	 Dariusz Tywoniuk \at
	 Institute of Mathematics and Computer Science, Wroclaw University of Technology, Wybrze\.ze Wyspia\'nskiego 27, 50-370 Wroc\l aw, Poland\\ \email{dariusz.tywoniuk@pwr.wroc.pl}}
	
	\date{Received: date / Accepted: date}

	\maketitle
	
	\begin{abstract}
		There are two main subjects in this paper. \\
		1) For a \tl\ \ds\ \xt\ we study the \tl\ entropy of its ``functional envelopes'' (the action of $T$ by left composition on the space of all continuous self-maps or on the space of all self-homeomor\-phisms of $X$). In particular we prove that for zero-dimensional spaces $X$ both entropies are infinite except when $T$ is equicontinuous (then both equal zero).\\
		2) We call \emph{Slovak space} any compact metric space whose homeomorphism group is cyclic and generated by a minimal homeomorphism. Using Slovak spaces we provide examples of (minimal) systems \xt\ with positive entropy, yet, whose functional envelope on homeomorphisms has entropy zero (answering a question posed by Kolyada and Semikina). Finally, also using Slovak spaces, we resolve a long standing open problem whether the circle is a unique non-degenerate continuum admitting minimal continuous transformations but only invertible: No, some Slovak spaces are such, as well.
		\keywords{Group of homeomorphisms \and Topological entropy \and Functional envelope \and Almost 1-1 extension \and Minimal space \and Continuum}
		\subclass{37B05 (primary) \and 37B40 \and 54H20 (secondary)}
	\end{abstract}
	
\begin{acknowledgements}
Part of the work was carried out during a visit of the first and third author in the Department of Mathematics, Faculty of Natural Sciences, Matej Bel University, Bansk\'a Bystrica, Slovakia. We gratefully acknowledge the hospitality of Matej Bel University.
\end{acknowledgements}
	
\maketitle

\section{Introduction}
A (\tl) \ds\ \xt, where $X$ is a compact metric space and $T:X\to X$ is a continuous transformation,
gives rise to several ``hyper-systems''. For example, $T$ acts naturally on $2^X$, the space of
all compact subsets of $X$ equipped with the Hausdorff metric. Likewise, $T$ extends to a continuous
(in the weak-star topology) operator on the compact space of all (Borel) probability measures on $X$,
by the formula $(T\mu)(A) = \mu(T^{-1}(A))$ (see e.g. \cite{BS}).

In the same spirit (although now the domains are not necessarily compact), several versions of so-called \emph{functional envelope} of a dynamical system have been considered. In~\cite{AuKS} the authors introduce and study the (non-compact) ``hyper-system'' $(C(X,X), F_T)$, whose phase space $C(X,X)$ consists of all continuous self-maps of $X$, and the transformation $F_T: C(X,X) \to C(X,X)$ is defined by $F_T(\varphi) = T \circ \varphi$ ($\varphi \in C(X,X)$). Notice that $C(X,X)$ is a \tl\ semigroup with respect to the uniform topology and $F_T$ is uniformly continuous (because $T$ is uniformly continuous).
Since $X$ is compact, the uniform metric agrees with the compact-open topology and it is also equivalent to the Hausdorff metric applied to the graphs of functions $X\to X$ (which are closed in $X\times X$). However, these two metrics are in general not uniformly equivalent, therefore the two versions, ``uniform" and ``Hausdorff", of the functional envelope (both addressed in \cite{AuKS}) may differ in some dynamical properties which depend on the metric (for example the \tl\ entropy in non-compact systems does).\footnote{On the other hand, it is easy to see that the essential dynamical properties of the functional envelopes do not depend on the metric applied on $X$.} In the present paper we focus exclusively on the uniform metric, nonetheless some of our results extend without much effort to the Hausdorff metric as well. Also note that the action of $F_T$ is by the composition with $T$ on the left. The action by the composition on the right is always a contraction (by which, throughout this note, we understand that it does not increase the distance) in the uniform metric on $C(X,X)$ (it is an isometry if $T$ is surjective), so it is not very interesting in the studies of entropy.

Now suppose additionally that the map $T$ is a homeomorphism. Then the space $H(X,X)$ of all self-homeomorphisms of $X$ is invariant under $F_T$ and $(H(X,X), F_T)$ becomes another kind of functional envelope (also addressed in \cite{KS}). Now $H(X,X)$ is a topological group, in general not compact and not even complete in the uniform metric. It is thus customary to consider this group with the equivalent, but complete, symmetric metric obtained as the sum of the uniform distance and the uniform distance between inverses. However, the second part of this metric is easily seen to be invariant under $F_T$. Thus, despite the fact that it makes the action of $F_T$ distal, it does not contribute to entropy (which is our main subject of interest), therefore, from our point of view it is sufficient to equip $H(X,X)$ with the uniform distance (the lack of completeness will not bother us at all).
\medskip

It is known that the \tl\ entropy of the ``hyper-system'' induced on probability measures behaves
quite radically: it equals zero if and only if $\htop(X,T)=0$, otherwise it is infinite
(see \cite{GW1}, \cite{GW2}). For the hyper-system on $2^X$ the si\-tua\-tion is different: If $T$ has positive entropy then the hyper-system has infinite entropy, otherwise the hyper-system may have zero, finite positive or infinite entropy (see \cite{KO}). In this paper we will be interested in the \emph{\tl} entropy of the functional envelopes. Since no other entropy will be addressed, in what follows we will skip
the adjective ``\tl''. Likewise, all \tl\ spaces in this note are \emph{metric}, hence also this adjective will be skipped.
\medskip

Let us gather some basic known facts and open problems concerning the entropy of the functional envelopes:
\begin{enumerate}
\item [(1)] The entropy of $(C(X,X), F_T)$ is always not smaller than that of \xt; to see this it
suffices to consider the subsemigroup of constant self-maps, which is conjugate to \xt.
\item [(2)] On the other hand, it is not known whether the same inequality holds for $(H(X,X),F_T)$, see~\cite{KS}.
\item [(3)] There are examples of zero entropy continuous maps on countable spaces and on the
interval, for which $(C(X,X), F_T)$ has infinite entropy \cite{AuKS}.
\item [(4)] On the other hand, analogous examples for $(H(X,X), F_T)$ are missing.
\item [(5)] S. Kolyada and J. Semikina \cite{KS} conjecture that the functional envelope $(C(X,X), F_T)$ has entropy either zero or infinity. At present it is known that the conjecture holds true for all Peano continua and for all compact spaces with continuum many connected components \cite{KS}. Otherwise the conjecture remains open. The same problem can be posed for the functional envelope $(H(X,X),F_T)$ in
case $T$ is a homeomorphism.
\end{enumerate}

\medskip

In this paper we resolve some questions of this type:
\begin{enumerate}
	\item[(a)] If $T$ is a self-homeomorphism of a compact zero-dimensional space $X$ then both functional
	envelopes $(C(X,X),F_T)$ and $(H(X,X),F_T)$ have infinite entropy except when $T$ is equicontinuous,
	in which case both entropies are equal to zero, see Theorem~\ref{T:zerodim}.
	This answers positively the question (2) and the Kolyada--Semikina conjecture (5) (also the version for 		
	$(H(X,X),F_T)$) for homeomorphisms on compact zero-dimensional spaces.
	\item[(b)] As an immediate consequence we get that there exist zero entropy homeomorphisms such that
	$(H(X,X),F_T)$ has infinite entropy, providing examples called for in (4). We also give
	a simple direct example, see the end of Section~\ref{S:zero-dim case}.
	\item[(c)] There exists a positive (even infinite) entropy homeomorphism $T$ of a compact space $X$
	such that $(H(X,X),F_T)$ has entropy zero, see Corollary~\ref{C:Tpositive FT zero}. This proves that
	the general question (2) has a 	negative answer.
\end{enumerate}

The item (c) above is solved using a new class of spaces, which we call \emph{Slovak spaces}\footnote{The definition was postulated by the Slovak member of the team of authors and the examples were constructed during the visit of the Polish authors in Slovakia. Besides, the notion of Polish spaces already exists, and we believe that Slovakia also deserves to have its own class. And finally, our example of a Slovak space resembles a country with many mountain ranges, just like Slovakia.}, defined by the combination of two properties: the existence of a minimal homeomorphism, say $T$, and nonexistence of homeomorphisms other than the powers of~$T$. For such spaces we prove that the functional envelope
$(H(X,X), F_T)$ always has entropy zero. Of course, it is completely unclear that such spaces exist, not to mention that we would like the system \xt\ to have positive (or even infinite) entropy. Thus, large part of the paper is devoted to delivering appropriate examples, see Section~\ref{S:Slovak}.

As a byproduct, nonetheless an interesting and important outcome of our work, we prove
that among Slovak spaces we find long searched for examples of non-degenerate continua of type (1,0) (admitting a minimal homeomorphism but not a minimal non-invertible map) other than the circle (see \cite{BKS} for the formulation of the problem).

\section{Preliminaries}

A {\it continuum} is a nonempty, compact and connected space. A continuum is {\it decomposable} if it is a union of two proper subcontinua, otherwise it is called {\it indecomposable}. The {\it composant of a point} $x$ is the union of all proper subcontinua of $X$ which contain $x$. By a \emph{composant in $X$} we will mean the composant of some point. 	The continuum is said to be \emph{irreducible between two different points $x,y$} if there is no proper subcontinuum containing both $x$ and $y$. We will say, for short, that $x,y$ is an \emph{irreducible pair}.

We recall some known facts on composants and irreducible pairs; they all can be found in (or easily deduced from) \cite[pp. 83, 97, 196--205]{Na}.
\begin{enumerate}
	\item[(C1)] Unless $X$ is degenerate, the composants are dense connected subsets of $X$ and they cover $X$.
 	\item[(C2)] A subset $A$ of $X$ contains an irreducible pair if and only if there exist proper composants, but none of them contains $A$.
	\item[(C3)] A decomposable continuum has either just one composant $X$ or three composants: $X$ and two proper composants that cover $X$.
	\item[(C4)] If $X$ is a non-degenerate indecomposable continuum then its composants are pairwise disjoint, hence form a partition of $X$ into sets of first category, and thus there are uncountably many of them.
\end{enumerate}

We make an additional observation (which is not included in \cite{Na}):

\begin{lemma}\label{compan}
If $\varphi:X\to Y$ is a continuous map between continua and $Y$ has more than one composant, then each composant $\alpha$ of $X$ is mapped either into a proper composant of $Y$ or onto $Y$.
\end{lemma}

\begin{proof} By (C2), if $\varphi(\alpha)$ is not contained in a proper composant then it contains an irreducible pair. A pair of preimages (of this pair) is contained in $\alpha$, hence in a subcontinuum
$C$ of $\alpha$. The image of $C$ is a continuum which contains an irreducible pair, so it is not proper, thus $\varphi(\alpha)  = Y$. \hfill $\square$
\end{proof}

A dynamical system $(X,T)$ is called {\it minimal} if every forward orbit is dense
(if $T$ is a homeomorphism it suffices that every full orbit is dense). It is well known that a minimal system contains neither subinvariant nor superinvariant (i.e., such that $T(A)\subset A$ or $T(A)\supset A$, respectively) nonempty proper closed subsets (see \cite[Lemma 3.10]{BOT}). It is also well known that if $X$ is connected then every minimal map $T$ is \emph{totally minimal}, i.e., all iterates $T^n$, except of identity, are minimal.

A space $X$ is called of \emph{type $(i,j)$}, where $i,j\in\{0,1\}$, depending on whether there exists a minimal homeomorphism on $X$ (then $i=1$, otherwise $i=0$) and whether there exists a minimal non-invertible continuous map on $X$ (then $j=1$, otherwise $j=0$). For example, the interval
is of type $(0,0)$ (due to the fixed point property), the circle is of type $(1,0)$ (and up to now it was the unique non-degenerate continuum known to be of this type), the two-torus is of type $(1,1)$ and the pinched torus (two points glued together) is of type $(0,1)$.

We take this opportunity to prove a simple, yet never observed fact concerning decomposable continua with
three composants.

\begin{theorem}\label{threecomposants}
If $X$ is a decomposable continuum with three composants then it admits no minimal maps (is of type $(0,0)$).\footnote{None of the zeros in type $(0,0)$ follows trivially from the periodic point property. A counterexample exists in \cite[Proposition 1]{HMP}: the continuum $X$ is formed by two concentric circles connected by a spiral (which is a continuous injective image of the real line), whose one ``tail" wraps around the smaller circle from outside and the other ``tail" approaches the larger circle from inside. This continuum has three composants (the complements of each circle and $X$).
Any irrational rotation of the circles extends easily to a self-homeomorphism of $X$ having no periodic points. It is also easy to find noninvertible continuous surjections without periodic points.}
\end{theorem}

\begin{proof}
Let $\alpha$ and $\beta$ be the two proper composants.
Define $A=\overline{X\setminus\alpha}$ and $B=\overline{X\setminus\beta}$ (both are nonempty). If $A=B=X$ then $\alpha$ and $\beta$ have empty interiors implying that every proper subcontinuum is nowhere dense
(by (C3) and (C2), every proper subcontinuum is a subset of either $\alpha$ or $\beta$). Since $X$, being decomposable, is a union of two such subcontinua, this is impossible. Thus, we have e.g. $A\neq X$. Let $T:X\to X$ be a minimal map. Every proper subcontinuum $C$ is mapped by $T$ to a proper subcontinuum (otherwise it would be a superinvariant nonempty proper closed set). Thus $T(\alpha)$ contains no irreducible pairs, hence $T(\alpha)\neq X$. By Lemma~ \ref{compan}, $\alpha$ is mapped into a proper composant (likewise, so is $\beta$). Since $T$ is surjective and $\alpha\cup\beta = X$, both proper composants cannot be mapped into the same one. Now it is easy to see that either $T(A) \supset A$ and $T(B) \supset B$, or $T(A) \supset B$ and $T(B) \supset A$. In either case $T^2(A) \supset A$. Since $A$ is nonempty, closed and proper, $T^2$ is not minimal, a contradiction with total minimality on the connected space~$X$. \hfill $\square$
\end{proof}

A dynamical system $(X,T)$ is an {\it extension} of $(Y,S)$, or $(Y,S)$ is a {\it factor} of $(X,T)$, if there exists a continuous surjection $\pi\colon X\to Y$ such that $\pi T=S\pi$. The map $\pi$ is called a {\it factor map}. The factor map $\pi$ is an {\it almost 1-1 factor map} and $(X,T)$ is an {\it almost 1-1 extension} of $(Y,S)$ if the set of singleton fibers, i.e., the set $\{x\in X: \, \pi^{-1}(\pi (x)) =x\}$ is residual in $X$ (it suffices that it is dense). A factor of a minimal system is minimal and an almost 1-1 extension of a minimal system is minimal (see e.g. \cite{D}).

A point $x\in X$ is (forward) {\it periodically recurrent}\footnote{In the literature, this property is often called \emph{regular recurrence}, however, we believe that \emph{periodic recurrence} is more informative.} if for every open neighborhood $U$ of $x$ there exists $k\in\na$ such that $T^{kn}(x)\in U$ for all $n\in\na$.

\emph{Adding machines} are infinite compact monothetic zero-dimensional groups. Since every adding machine is homeomorphic to the Cantor set, we will denote both by the letter $\mathfrak C$. By fixing a topological generator $c_0$ and defining $h:\mathfrak C\to \mathfrak C$ by $h(c) = c\oplus c_0$ (where $\oplus$ denotes the addition in the adding machine) we obtain a minimal equicontinuous dynamical system $(\mathfrak C,h)$ which is called an \emph{odometer}. Odometers are characterized as minimal systems in which every point is periodically recurrent (see e.g. \cite{D}).

{\it Solenoids} are quotient spaces of the product $[0,1]\times \mathfrak C$ of the interval with the Cantor set, with respect to the relation identifying the points $(1,c)$ and $(0,h(c))$, where $h$ is as described above. {\it Generalized solenoids} are quotient spaces of the product $[0,1]\times \mathfrak C$ of the interval with the Cantor set, with respect to the relation identifying the points $(1,c)$ and $(0,h(c))$, where $h$ is any minimal homeomorphism of the Cantor set $\mathfrak C$. Generalized solenoids are indecomposable continua.

\medskip
Since in this paper we will be addressing \tl\ entropy not only in compact \ds s, but also in non-compact ones, let us recall the definition of topological entropy in this more general setup (see, e.g., \cite{W}).
Let $(Z;\varrho)$ be a space and $T:Z\to Z$ be uniformly continuous. For every $n\geq 1$ the function
$\varrho_{n}(x,y):=\max\{\varrho(T^j(x),T^j(y)):0\le j\le {n-1}\}$ defines a metric on $Z$ equivalent with $\varrho$. Fix an $n\ge 1$ and $\epsilon >0$ and let $K$ be a compact set in $Z$.
A subset $E\subset K$ is called \emph{$(n,\epsilon)$-separated\/}, if
for any two distinct points $x,y\in E$, $\varrho_{n}(x,y)>\epsilon$.
Since $K$ is compact, $E$ is finite.
Denote by $\sep(n,\epsilon; K)$  the maximal cardinality of an
$(n,\epsilon)$-separated set in $K$.
The \emph{topological entropy of $T$ on $K$}
is defined by
$\htop (T, K):=\lim_{\epsilon \to 0}\limsup_{n\to\infty}
\frac{1}{n}\log \sep (n,\epsilon; K)$
and the \emph{topological entropy} of $T$ is defined by
$\htop (T):= \sup_K \htop (T,K)$
where $K$ ranges over all compact subsets of $Z$.
\medskip

If $\varphi$ is a map, we denote its graph by $\graph(\varphi)$. If $x$ is a continuity/discontinuity point of $\varphi$, also the point $(x,\varphi(x)) \in \graph(\varphi)$ will sometimes be called a continuity/discontinuity point of $\varphi$. We hope this will cause no misunderstandings.

\section{Entropy of functional envelopes; the homogeneous and the zero-dimensional cases}\label{S:zero-dim case}

We begin by solving (positively) the problem (2) for homogeneous spaces. Recall that $X$ is \emph{homogeneous} if for any two points $x,y\in X$ there exists a homeomorphism $f: X\to X$ such that $f(x)=y$. If $X$ is compact, we have even more: for each $x\in X$ there exists a \emph{compact} family $\mathcal F_x \subseteq H(X,X)$ such that for every $y\in X$ there is an $f\in \mathcal F_x$ satisfying $f(x)=y$. Indeed, by Effros Theorem \cite[Theorem~2.1]{E} (see also \cite{Un}), the evaluation map $E_{x}\colon H(X,X)\to X$, given by $E_{x}(h)=h(x)$, is both continuous and open. Since $X$ is homogeneous, this map is also surjective. Moreover, $H(X,X)$ is complete in the equivalent symmetric uniform metric. Now the existence of a compact family $\mathcal F_x$ follows directly from \cite[Lemma 6]{B}.

\begin{proposition}\label{P:homo0}
Let $T: X\to X$ be a homeomorphism of a homogeneous compact space. Then the entropy of the functional envelope $(H(X,X), F_T)$ is at least as large as that of \xt.
\end{proposition}

\begin{proof}
Fix some $x\in X$ and let $E_{x}: H(X,X)\to X$ be the evaluation map. Clearly $E_{x}$ is a contraction (i.e., does not increase the uniform distance), $E_{x}(\mathcal F_x)=X$ and $E_{x}\circ F_T=T\circ E_{x}$. This implies that for every $n$ and $\epsilon$ there are at least as many $(n,\epsilon)$-separated (under $F_T$) homeomorphisms in the compact set $\mathcal F_x$ as there are $(n,\epsilon)$-separated (under $T$) points in $X$. This clearly implies the assertion. \hfill $\square$
\end{proof}

The next theorem resolves completely the question about the entropy of the functional envelope on the group of homeomorphisms in compact zero-dimensional spaces. Recall that a homeomorphism $T:X\to X$ is called \emph{equicontinuous} if the family of all forward and backward iterates of $T$ is equicontinuous. We will be needing the following characterization of equicontinuous homeomorphisms on compact zero-dimensional spaces. Because the proof of this ``folklore'' fact is hard to find, we provide it below.

\begin{lemma}\label{L:equi}
Let $X$ be a compact zero-dimensional space. A homeomorphism $T:X\to X$ is equicontinuous if and only if the
\ds\ \xt\ is a union of odometers and periodic orbits, if and only if every point $x\in X$ is periodically recurrent under $T$.
\end{lemma}

\begin{proof}
If $T$ is an equicontinuous homeomorphism then it is distal, hence the space is a union of minimal sets
(see e.g. \cite{Aus}). By the Halmos--von Neumann Theorem, any minimal equicontinuous system is the rotation of a compact monothetic group, and every zero-dimensional compact monothetic group is either finite cyclic or an adding machine (i.e., the corresponding \ds\ is an odometer). This obviously implies that each point in the whole system is periodically recurrent.

Now suppose all points are periodically recurrent. Let $\mathcal P$ be a partition of $X$ into finitely many closed-and-open (we will say \emph{clopen}) sets of diameter at most $\epsilon$. The collection of all $\mathcal P$-names of the points of $X$ is closed and shift-\inv, so it is a subshift. Suppose that the collection of the $\mathcal P$-names
is infinite. Then, by \cite[Corollary on page 63]{BW}, there are two different and forward asymptotic $\mathcal P$-names. Taking disjoint neighborhoods of these names (in the symbolic space), we see that at least one of them fails to be periodically recurrent (a contradiction occurs along common multiples of the periods with which the points visit their selected neighborhoods).
So, there are only finitely many $\mathcal P$-names, i.e., the partition $\mathcal P$ generates (via the dynamics) a finite sigma-algebra with clopen atoms.
Letting $\delta>0$ be the minimal distance between these atoms, we see that any points $x,y\in X$ with $d(x,y)<\delta$ satisfy, for each $n\ge 0$ the condition: $T^nx$, $T^ny$ belong to the same atom of $\mathcal P$. Thus $d(T^nx$, $T^ny)<\epsilon$. The same is true for $T^{-1}$, concluding the proof of equicontinuity. \hfill $\square$
\end{proof}

\medskip

\begin{theorem}\label{T:zerodim}
Let $T:X\to X$ be a self-homeomorphism of a compact zero-dimensional space. Then the entropies of $(C(X,X),F_T)$ and $(H(X,X),F_T)$ are either both zero or both infinite. They are equal to zero if and only if $T$ is equicontinuous.
\end{theorem}

\begin{proof} If $T$ is equicontinuous then there exists an equivalent metric on $X$ for which $T$ is an isometry (an \inv\ metric). By compactness, the two metrics are uniformly equivalent. Thus, the corresponding uniform metrics on $C(X,X)$ (and $H(X,X)$) are also uniformly equivalent, implying that the entropies of $F_T$ will not change if we change the metric. For the invariant metric $F_T$ is also an isometry, hence the entropies of $(C(X,X),F_T)$ and $(H(X,X),F_T)$ are equal to zero.

Now suppose $T$ is not equicontinuous. Then, by the proof of Lemma~\ref{L:equi}, there exists a finite partition $\mathcal P$ of $X$ into clopen sets, such that the subshift on the $\mathcal P$-names contains a nonperiodic point. This implies that the complexity $c(n)$ of this subshift (the number of blocks of length $n$) satisfies $c(n)\ge n+1$ (see \cite{HM}).\footnote{It is known that only periodic subshifts have bounded complexity, the next lowest possibility is $c(n)=n+1$ and occurs e.g. in the subshift consisting of $\{0,1\}$-\sq s with at most one symbol 1 and, for minimal subshifts, in Sturmian systems.} Let $\epsilon$ be the smallest distance between the atoms of $\mathcal P$. The complexity yields that given $n$ there exist (more than) $n$ $(n,\epsilon)$-separated points for $T$. Further, there exist (more than) $n!$ permutations of these points. If every such permutation could be extended to a homeomorphism of the whole space, we would have obtained $n!$ homeomorphisms which are $(n,\epsilon)$-separated in the uniform distance under $F_T$. We will use this principle to create our compact family of homeomorphisms with hyper-exponential rate of separation.

Fix a decreasing to zero \sq\ $\delta_k$ and for each $k$ let $\mathcal Q_k$ be a finite partition into
clopen sets of diameter not exceeding $\delta_k$. Let $q_k$ denote the cardinality of $\mathcal Q_k$. Let $n_k$
be a large integer (how large we will specify later). As we have noticed earlier, we can select $n_k$ points in $X$ which are $(n_k,\epsilon)$-separated under $T$. Some element $Q$ of the partition
$\mathcal Q_k$ contains at least $\frac{n_k}{q_k}$ of these points. We would like to be able to permute
these points by homeomorphisms.

Let $X_0$ denote the (possibly empty) closure of the set of all isolated points of $X$ and let
$X_1$ be the (also possibly empty) rest of $X$. If at least half of the $\frac{n_k}{q_k}$ $(n_k,\epsilon)$-separated points contained in $Q$ lie in $X_0$, we can replace them by nearby isolated points, so that the new points remain $(n_k,\epsilon)$-separated and lie in $Q$. Now, every permutation of these points can be extended to a homeomorphism which is different from identity only on $Q$. In the remaining case, at least half of our points lie in $X_1$ and thus in a small Cantor set contained in $Q$ and clopen in $X$. Again, these points are permutable by homeomorphisms differing from identity only on $Q$.

So, in either case, we have found $(\frac{n_k}{2q_k})!$ homeomorphisms permuting some $\frac{n_k}{2q_k}$ $(n_k,\epsilon)$-separated points, and differing from identity only on a set of diameter not exceeding $\delta_k$. Let us denote this family of homeomorphisms by $K_k$. Notice that they are all $(n_k,\epsilon)$-separated under $F_T$. It is obvious that the families $K_k$ converge uniformly to identity as $k$ grows, so the collection $\{\mathsf{id}\}\cup\bigcup_k K_k$ is compact.

The entropy of the functional envelope $(H(X,X),F_T)$ (and thus also that of $(C(X,X),F_T)$) is hence estimated from below by the number $\limsup_k h_k$ where $h_k=\frac1{n_k}\log((\frac{n_k}{2q_k})!)$.

Since $n!>(\frac n2)^{\frac n2}$ (and thus $\log(n!)>\frac n2(\log \frac n2)$), we get
$h_k > \tfrac 1{4q_k}(\log n_k - \log 4q_k)$. It suffices to choose $n_k$ so that $\log n_k$ grows essentially faster than $q_k$ to get infinite entropy. \hfill $\square$
\end{proof}

\begin{question}
Is an analogous theorem true for continuous non-invertible self-maps $T$ on compact zero-dimensional spaces and the functional envelope $(C(X,X), F_T)$?
\end{question}

As promised, we now give a simple direct example of a zero-entropy homeomorphism \xt\ with infinite entropy of $(H(X,X),F_T)$.

\begin{example}
Consider the action of the ``+1'' map on the one-point compactification of the integers.
The space is compact zero-dimensional, $T$ has entropy zero and is not equicontinuous. Theorem \ref{T:zerodim} implies that $(H(X,X),F_T)$ has infinite entropy.
\end{example}

\section{Slovak spaces}\label{S:Slovak}

In most examples of infinite compact spaces admitting a minimal homeomorphism (the circle, the torus, the Cantor set) there are usually uncountably many such homeomorphisms. On the other hand, some other spaces admit no minimal homeomorphisms. Do there exist ``intermediate'' infinite compact spaces in the sense that they admit some, but at most countably many, minimal homeomorphisms?

It is known, see \cite{dG}, that for every abstract group $G$ there exists a topological space $X$ such that $H(X,X)\simeq G$, i.e., $H(X,X)$ is algebraically isomorphic to $G$. Such a space $X$ always exists in the class of one-dimensional, connected, locally connected, complete spaces and always exists in the class of compact, connected, Hausdorff (not necessarily metrizable) spaces. (However, such a space need not exist in the class of compact metric spaces, because a compact metric space has cardinality at most $\mathfrak{c}$, while there are groups of arbitrary cardinalities.) Moreover, as proved in \cite{dGW}, if $G$ is countable then $X$ can be chosen to be a Peano continuum of any positive dimension.
In particular, there is a Peano continuum $X$ with $H(X,X)$ being the trivial group (then $X$ is called a \emph{rigid space for homeomorphisms}). Also, there is a Peano continuum $X$ such that $H(X,X)\simeq \mathbb Z$, i.e., $H(X,X) =\{T^n:\, n\in \mathbb Z\}$, the elements of $H(X,X)$ being pairwise distinct. In view of these facts, we are interested in whether there exists a compact (metric) space $X$ such that $H(X,X)\simeq \mathbb Z$ and the generating homeomorphism $T$ is minimal.
\smallskip

We adopt the following definition.

\begin{definition}
A compact space $X$ is called a \emph{Slovak space} if it has at least three elements, admits a minimal homeomorphism $T$ and $H(X,X)=\{T^n:n\in\Z\}$.
\end{definition}

\begin{theorem}\label{cardinality}
If $X$ is a Slovak space then the cyclic group $H(X,X)$ is infinite (i.e., isomorphic to $\Z$), and all its elements, except identity, are minimal homeomorphisms.
\end{theorem}

\begin{proof}
The first condition in the definition eliminates two trivial cases: the one-point space and the two-point space. It is elementary to check that for any larger finite space $X$, $H(X,X)$ (which is the group of all permutations of a finite set) has at least two generators, so $X$ is not Slovak. Any infinite compact space admitting a minimal homeomorphism does not have isolated points (the complement of the set of isolated points would then be nonempty, closed, proper and invariant), and thus any Slovak space is perfect (in particular, its cardinality equals $\mathfrak{c}$). If $T$ is a generator of $H(X,X)$ then the powers $T^n$ are different for different exponents (moreover, each pair of powers differs at every point), otherwise the system $(X,T)$ would contain periodic orbits, contradicting minimality. Thus $n\mapsto T^n$ is an isomorphism between $\Z$ and $H(X,X)$.

Suppose that $T^n$ for some $n>1$ (or $n<-1$) is not minimal. This is only possible when $X$ decomposes into pairwise disjoint clopen sets $X_0,X_1,\dots,X_{k-1}$ ($k>1$ being a divisor of $n$) which are cyclically permuted by $T$. Then each $X_i$ is infinite and invariant under $T^k$. The homeomorphism defined as $T^k$ on $X_0$ and identity on $X\setminus X_0$ is not the identity ($T^k$ must not have fixed points), but it has fixed points, hence it is not a nonzero power of $T$, either. So, $X$ is not a Slovak space. \hfill $\square$
\end{proof}

Notice that the above theorem implies that for every Slovak space there are exactly two possible choices of the generator of $H(X,X)$: a $T$ and $T^{-1}$ (both minimal). The next theorem establishes even more precisely the basic topological properties of Slovak spaces.

\begin{theorem}\label{continuum}
Every Slovak space is a non-degenerate continuum.
\end{theorem}

\begin{proof}
We need to show that every Slovak space is connected. Fix a Slovak space $X$ and a  homeomorphism $T$ (of course minimal) which generates $H(X,X)$. Suppose that $X$ is not connected. Then $X$ is a union of
two nonempty, disjoint (hence proper), clopen subsets: $X=A\cup B$. Let $C=B\setminus T(B)$. Since $T$ is a homeomorphism, the set $C$ is again clopen. By minimality, $B$ is not superinvariant, hence $C$ is nonempty. Clearly, so defined set satisfies $C\cap T(C)=\emptyset$. Let us define a function $S\colon X\to X$ by the formula
$$S(x) = \begin{cases} T(x), & \text{if $x \in C$,} \\
                   T^{-1} (x)                                         , & \text{if $x\in T(C)$,}\\
									 x,& \text{if $x\in X\setminus(C\cup T(C))$.}
       \end{cases}
$$
It is obvious that $S$ is a self-homeomorphism of $X$. Notice that any element of the set $C$ is periodic with period two. Thus $S$ is not minimal and different from identity. In view of the preceding theorem, such an $S$ must not exist on a Slovak space, a contradiction. \hfill $\square$
\end{proof}

A priori it is not clear that Slovak spaces exist. We prove the existence of such spaces in the class of one-dimensional continua. We remark that minimality is completely inessential in our considerations of the entropy of the functional envelope. Instead of minimality we should require that the generating homeomorphism has positive \tl\ entropy. (In fact, we will provide examples with both minimality and positive entropy satisfied.) Minimality requirement makes Slovak spaces more interesting and useful for the future application to types of minimality (see Section \ref{lastsection}). Also notice that for Slovak spaces, positivity or finiteness of the entropy of $(X,S)$ does not depend on the choice of $S\in H(X,X)$, $S\neq\mathsf{id}$. Thus we can speak about \emph{zero, finite positive} and \emph{infinite entropy} Slovak spaces. Moreover, we can accept the entropy of the generating homeomorphism $T$ (or of $T^{-1}$, which is the same) to be called \emph{the entropy of the Slovak space}.

\medskip
Before we provide an evidence for the existence of Slovak spaces, we prove that regardless of the entropy of the Slovak space, the functional envelope $(H(X,X), F_T)$ has always entropy zero.

\begin{proposition}\label{P:Slovak entropy}
If the group $H(X,X)$ is countable then $(H(X,X), F_T)$ has entropy zero for any $T\in H(X,X)$.
\end{proposition}

\begin{proof}
In any compact space, the group of homeomorphisms with the uniform metric is Polish, being complete in the equivalent symmetric uniform metric. It is also homogeneous hence either discrete or perfect. However, a perfect Polish space is uncountable. Therefore if the group of homeomorphisms is countable, it is discrete. Since only finite subsets of a discrete space are compact, every action on a discrete space has entropy zero. \hfill $\square$
\end{proof}

We will now pass to an effective production of a Slovak space (in fact we will construct a family of such spaces). The following lemma will be useful in the construction.

\begin{lemma}\label{L:UC}
Let $T:X\to X$ be a homeomorphism of a compact space $X$. Fix some $x_0\in X$ not periodic under $T$ and let $f:X\setminus\{x_0\}\to [0,1]$ be continuous (at every point of its domain). Let $F = \sum_{n\in\Z} a_n f\circ T^n$, where the coefficients $a_n$ are all strictly positive, $\sum_{n\in\Z}a_n=1$ and with the ratios $\frac{a_{n-1}}{a_n}$ bounded from above ($F$ is defined on $X'$, the complement of the orbit of $x_0$). Then the mapping $(x,F(x))\mapsto(Tx,F(Tx))$ is a uniformly continuous homeomorphism of the graph of $F$.
\end{lemma}

\begin{proof} Since $x_0$ is not periodic, all the points $T^n x_0$ are distinct, moreover, $X'$ is nonempty (the orbit of $x_0$ is infinite, countable and homogeneous, so it cannot be compact) and $T$-\inv.
By summability of the \sq\ $a_n$, $F$ is a continuous function on its domain. Now, since $T$ and $T^{-1}$ are uniformly continuous on $X'$, it is clear that the specified map is a homeomorphism of the graph of $F$. We only need to verify its uniform continuity.

Let $d$ be the metric in $X$. Let $A\ge 1$ denote an upper bound for $\frac{a_{n-1}}{a_n}$. Fix $\epsilon>0$. Let $n_0>0$ be such that $\sum_{n\in\Z\setminus[-n_0,n_0-1]}a_n<\frac\epsilon{5A}$. Let $r$ be such that the $2r$-balls $B(T^nx_0,2r)$ are disjoint for $|n|\le n_0$. Denote by $U_n$ the analogous $r$-balls in the reversed order: $U_n=B(T^{-n}x_0,r)$, $|n|\le n_0$. Note that these balls are separated by distances at least $r$. For easier writing denote $f_n = f\circ T^n$, remembering that $f_n(Tx) = f_{n+1}(x)$. Since each of the functions $f_n$ with $|n|\le n_0$ is continuous (hence uniformly continuous) on the complement of $U_n$, and we consider only finitely many functions, there exists $\delta<\min\{r,\frac\epsilon{5A}\}$ such that if $d(x,x')<\delta$ and both points $x,x'$ lie outside $U_n$ then $|f_n(x) - f_n(x')|<\frac\epsilon{5A(2n_0+1)}$. Because $T$ is uniformly continuous on $X$, we can choose $\delta$ small enough to also satisfy $d(x,x')<\delta\implies d(Tx,Tx')<\epsilon$. This concludes the definition of $\delta$ for our given $\epsilon$. Suppose $(x,F(x))$ and $(x',F(x'))$ are $\delta$-close (at each coordinate) in $X'\times[0,1]$. In particular, $Tx$ and $Tx'$ are $\epsilon$-close. All we need to show is that $|F(Tx)-F(Tx')|<\epsilon$.

First we notice that since $\delta<r$, the set $\{x,x'\}$ can intersect at most one of the balls $U_n$
with $|n|\le n_0$. In such case let ${\bar n}$ denote the corresponding index. Otherwise ${\bar n}$ is not defined.  Next, we write
\begin{multline*}
|F(Tx)-F(Tx')| =\left|\sum_{n\in\Z}a_n\bigl(f_n(Tx)-f_n(Tx')\bigr)\right|= \left|\sum_{n\in\Z}a_n\bigl(f_{n+1}(x)-f_{n+1}(x')\bigr)\right|\le\\ \sum_{n\in\Z}a_n|f_{n+1}(x)-f_{n+1}(x')| = \sum_{n\in\Z}a_{n-1}|f_n(x)-f_n(x')|=
\end{multline*}
\begin{multline*}
\sum_{n\in \Z\setminus[-n_0,n_0]}a_{n-1}|f_n(x)-f_n(x')| \ \ + \\
\sum_{n\in [-n_0,n_0]\setminus\{\bar n\}}a_{n-1}|f_n(x)-f_n(x')|\ \ + \\
a_{{\bar n}-1}|f_{\bar n}(x)-f_{\bar n}(x')|
\end{multline*}
(if ${\bar n}$ is not defined the last term simply does not occur).
The first sum is smaller than $\frac\epsilon{5A}$, because the indices for $a_{n-1}$ range within $\Z\setminus[-n_0-1,n_0-1]$ what is contained in $\Z\setminus[-n_0,n_0-1]$. So is the second one, because it has at most $2n_0+1$ terms, each not exceeding $\frac\epsilon{5A(2n_0+1)}$. Only the last term cannot be estimated so easily.

To do this, we need the inequality $|F(x)-F(x')|<\frac\epsilon{5A}$ (which follows from the assumption that the pairs $(x,F(x))$, $(x',F(x'))$ are $\delta$-close at each coordinate).

We have
$$
\tfrac\epsilon{5A} > |F(x)-F(x')| =
$$
\begin{multline*}
\Bigl|\sum_{n\in \Z\setminus[-n_0,n_0]} a_n\bigl(f_n(x)-f_n(x')\bigr) \ \ + \\
\sum_{n\in[-n_0,n_0]\setminus\{\bar n\}}a_n\bigl(f_n(x)-f_n(x')\bigr)\ \ + \\
a_{\bar n}\bigl(f_{\bar n}(x)-f_{\bar n}(x')\bigr)\Bigr|.
\end{multline*}
			
Analogously as before, the first and second sums are smaller (in absolute value) than $\frac\epsilon{5A}$.
This easily yields that
$$
a_{\bar n}|f_{\bar n}(x)-f_{\bar n}(x')|<\tfrac{3\epsilon}{5A},
$$
hence, multiplying both sides by $\frac{a_{\bar n-1}}{a_{\bar n}}$ (bounded by $A$), we get
$$
a_{{\bar n}-1}|f_{\bar n}(x)-f_{\bar n}(x')|<\tfrac{3\epsilon}{5}.
$$

We can now return to the estimation of $|F(Tx)-F(Tx')|$, and get that it is smaller than
$\tfrac{\epsilon}{5A}+\tfrac{\epsilon}{5A}+\tfrac{3\epsilon}{5} \le \tfrac{5\epsilon}{5} = \epsilon$ 
(recall that $A\ge 1$).  \hfill $\square$
\end{proof}

\begin{corollary}\label{C:pre-Slovak}
If the inverted ratios $\frac{a_n}{a_{n-1}}$ are also bounded from above, then the map $(x,F(x))\mapsto(Tx,F(Tx))$ extends to a homeomorphism $\widetilde T$ of the closure of the graph of $F$,
which we will denote by $\widetilde F$. If \xt\ is minimal (not periodic) then the \ds\ $(\widetilde F,\widetilde T)$ is a minimal almost 1-1 extension of \xt.
\end{corollary}

\begin{proof}
Applying Lemma \ref{L:UC} to $T^{-1}$ we get that $(x,F(x))\mapsto(T^{-1}x,F(T^{-1}x))$ is uniformly continuous as well. Therefore, the map $(x,F(x))\mapsto(Tx,F(Tx))$ extends to a homeomorphism $\widetilde T$ of $\widetilde F$. Clearly, $(\widetilde F,\widetilde T)$ is a \tl\ extension of the system $(\overline{X'},T)$; the projection to the first coordinate serves as the corresponding factor map. Moreover, this projection is injective at all points of $X'$, because these are continuity points of $F$, hence the sections of $\widetilde F$ at such points are singletons. This implies that the projection from $\widetilde F\to \overline{X'}$ is an almost 1-1 extension (singleton fibers are dense in $\widetilde F$). Finally, if \xt\ is minimal nonperiodic, then $\overline X'=X$. Now we use the general fact that an almost 1-1 extension of a minimal system is minimal, to deduce minimality of $(\widetilde F,\widetilde T)$. \hfill $\square$
\end{proof}

Let $h\colon Y\to Y$ be a homeomorphism of a compact space $Y$. Let $X$ be the compact space obtained from $Y\times [0,1]$ by identifying the pairs $(y,1)$ with $(h(y),0)$ (for every $y\in Y$).
The \emph{suspension flow over $h$} is the flow $(\phi_t)_{t\in\R}$ defined on $X$ by the formula
$$
\phi_{t}(y,s) = (h^{\lfloor t+s\rfloor}(y), \{t+s\}),
$$
where $\lfloor\cdot\rfloor$ and $\{\cdot\}$ denote the integer and fractional parts of a real number, respectively. A flow is called \emph{minimal} if every its orbit is dense in the phase space. Clearly,
the suspension flow over a minimal homeomorphism is minimal. Recall, that the \tl\ entropy of a flow is
defined as the \tl\ entropy of the time-one map.

\medskip
\begin{theorem}\label{T:slovak}
There exist Slovak spaces. The entropies of Slovak spaces exhaust the interval $[0,\infty]$.
\end{theorem}

\begin{proof}
We start with the Cantor set $\mathfrak C$ and a minimal homeomorphism $h:\mathfrak C\to \mathfrak C$.
The suspension flow over $h$ is a minimal flow $\phi_t$ whose phase space is the generalized solenoid $X$ induced by $(\mathfrak C,h)$ (see above). By~\cite{Eg} (see also \cite{F}) there exists $t_0\in\mathbb R$ such that $T=\phi_{t_0}$ is a minimal homeomorphism on $X$. We are going to apply Lemma~\ref{L:UC} and Corollary~\ref{C:pre-Slovak} to the system $(X,T)$. To this end, we fix a point $x_0\in X$ and we shall define a function $f:X\setminus\{x_0\}\to [0,1]$. Let $g:[-\frac12,\frac12]\to [0,1]$ be defined as
$$
g(x) = \begin{cases} \frac{1}{2}\left(1-\cos \frac{\pi}{x} \right), & \text{if $x \in [-\frac12, 0)$,} \\
                   0                                           , & \text{if $x\in[0,\frac12]$.}
       \end{cases}
$$
Note that $g$ is continuous except at $0$ where it is discontinuous from the left and continuous from the right (it is a version of so-called \emph{topologist's sine curve}). Let $\gamma$ denote the composant of $X$ which contains $x_0$. It is clear that each composant equals the orbit under the suspension flow $\phi_t$ of every its element, in particular $\gamma$ is the orbit of $x_0$. This determines a natural continuous bijection $p:\mathbb R\to \gamma$ by the formula $p(t)=\phi_{t\cdot t_0}(x_0)$ ($\Z$ maps onto the $T$-orbit of $x_0$). Let $J= p([-\frac12, \frac12]\setminus\{0\})$ and define $f_J: J \to [0,1]$ as $g \circ  p^{-1}$. Now let $f: X\setminus \{x_0\}\to [0,1]$ be any continuous extension of $f_{J}$ (we apply Tietze's Extension Theorem to $X\setminus \{x_0\}$ and its closed subset $J$). Note that $f$ and $x_0$ satisfy the assumptions from Corollary~\ref{C:pre-Slovak}.

Next, let $F = \sum_{n\in\Z} a_n f\circ T^n$ where the coefficients $a_n$ satisfy the assumptions from both Lemma~\ref{L:UC} and Corollary~\ref{C:pre-Slovak}. The closure $\widetilde F\subset X\times[0,1]$ of the graph
of $F$ is going to be our desired Slovak space. By Corollary~\ref{C:pre-Slovak}, the mapping $(x,F(x))\mapsto(Tx,F(Tx))$ defined on the graph of $F$ extends to a homeomorphism $\widetilde T :\widetilde F \to \widetilde F$, the system $(\widetilde F,\widetilde T)$ is a minimal almost 1-1 extension of \xt. It remains to show that $H(X,X)=\{\widetilde T^n:n\in\Z\}$.
\smallskip

Given $n\in \Z$, we can write $F= a_n f\circ T^n+\sum_{k\neq n}a_k f\circ T^k$ and we notice that the first
term is undefined at $T^{-n}x_0$ with the ``topologist's sine curve'' type of discontinuity along the main composant, while the remaining series represents a function continuous at this point. So, $F$ has the same ``topologist's sine curve'' type of discontinuity at every point of the orbit $O(x_0)$, and clearly it is continuous everywhere else. The almost 1-1 factor map $\pi:\widetilde F\to X$ coincides with the projection on the first coordinate and it is 1-1 at all continuity points of $F$, while $\pi^{-1}(T^n(x_0))$ is a closed interval $W_n$ (in fact $W_n = \{x_0\} \times [v_n, v_n + a_{-n}]$, where $v_n = \sum_{m\neq -n} a_m f \circ T^m(x_0)$).

Let us observe the path-components of $\widetilde F$. It is easy to see that they are the following: for points
in $\pi^{-1}(X\setminus\gamma)$ they are homeomorphic (via $\pi$) to the composants of $X$, and are continuous injective images of the real line. The set $\pi^{-1}(\gamma)$ breaks into countably many path-components, each being a continuous injective image of the closed half-line. It starts with the interval $W_n$ on the closed end and ends with the open ``topologist's sine curve'', adjacent to (containing in its closure) the next interval $W_{n+1}$. Let $Z$ be the collection of the closed
end-points of the path-components contained in $\pi^{-1}(\gamma)$ (see Figure~1) and note
that $Z$ is in fact an orbit under $\widetilde T$.

\begin{figure}[ht]\label{component}
\begin{center}
\includegraphics[width=12.5cm]{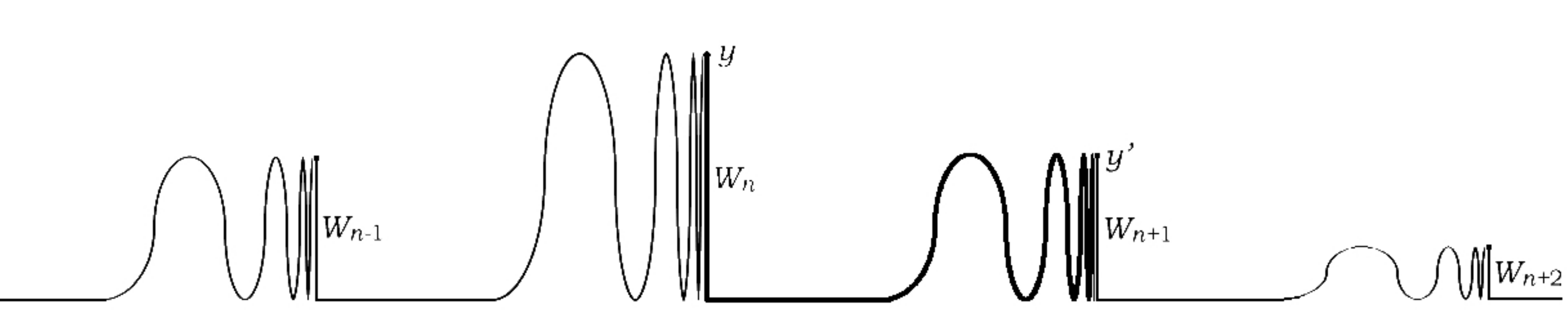}
\caption{\small\label{F:one}A piece of $\pi^{-1}(\gamma)$ with three complete path-components (one of them is marked with a thicker line). A point $y\in Z$ and its successor $y'\in Z$ are also marked. (The continuous perturbations caused by further terms of the sum defining $F$ are ignored.)}
\end{center}
\end{figure}

Let us call the \emph{successor} of $y\in Z$ the unique point $y'\in Z$, different from $y$, which belongs to the closure of the path-component of $y$. Obviously, the successor of $y$ is $\widetilde Ty$, but it is crucial for us that the successor relation is defined without referring to $\widetilde T$, in purely topological terms.

So, let now $\varphi$ be an arbitrary self-homeomorphism of $\widetilde F$. It sends path-components to path-components, preserving the type (real line or closed half-line), hence it also preserves the set $Z$ of
the closed ends of the half-line type components. Moreover, it preserves the successor relation on $Z$. This easily implies that on $Z$, $\varphi$ coincides with some iterate $\widetilde T^n$ ($n\in\Z$), and because $Z$ is dense in $\widetilde F$ (being an orbit under a minimal map), $\varphi=\widetilde T^n$, which completes the proof that $\widetilde F$ is a Slovak space.
\smallskip

Since $(\widetilde F,\widetilde T)$ is an almost 1-1 extension of \xt\ which is 1-1 except on a countable set, it has the same \tl\ entropy as \xt\footnote{An almost 1-1 extension which is 1-1 except on a countable set is \emph{principal} (preserves the Kolmogorov--Sinai entropy of every \im). It is so, because the countable set in the base system has measure zero for every \im\ (every such measure is nonatomic, as long as the base system has no periodic points), and so is the measure of its preimage in the extension system. So, the extension is 1-1 up to measure, and the extension system is isomorphic to the base system, in particular has the same K--S entropy (again, this is true for every \im). By the variational principle, every principal extension also preserves the \tl\ entropy.}, which equals $t_0\htop(h)$. Moreover, if $h$ is (for example) \tl ly weakly mixing, then every irrational $t_0$ produces minimal $\phi_{t_0}$.\footnote{The values of $t_0$ for which the map $\phi_{t_0}$ is not minimal are precisely such that $\exp(2\pi it_0)$ belong to the multiplicative subgroup of the torus $S_1$ generated by the roots of unity and the \tl\ eigenvalues of $h$. Thus the property that all irrational numbers $t_0$ produce minimal $\phi_{t_0}$ holds not only for \tl ly weakly mixing minimal homeomorphisms, but also for those which have only rational \tl\ eigenvalues.} Since there are \tl ly weakly mixing minimal Cantor systems with arbitrary entropy in $[0,\infty]$\footnote{It is known (\cite{Leh}) that every aperiodic ergodic system has a topologically mixing strictly ergodic model.}, we deduce that the entropy of Slovak spaces can have any value within this range. \hfill $\square$
\end{proof}

\begin{remark}
Our Slovak space $\widetilde F$, being a continuum with uncountably many composants, is indecomposable (by the property (C3), see also (C4)). The system $(\widetilde F,\widetilde T)$ is a \tl\ extension of an irrational circle rotation.
\end{remark}

\begin{question}
Is every Slovak space $X$ an indecomposable continuum?
\end{question}

Theorem \ref{threecomposants} eliminates continua with three composants, so, in view of (C3), the above question concerns only the one-composant continua.

\begin{remark}\label{circ}
An analogous construction applied to the circle (in place of the generalized solenoid) and irrational rotation (in place of the map $T$), which produces a one-composant continuum, does not work. One obtains a graph of a map $F$ defined on the circle with the orbit of a selected point $x_0$ removed. Now, there are many homeomorphisms of the circle (not just powers of the rotation) which preserve this orbit (as a set), and many of them can be lifted and then extended continuously to the closure of the graph of $F$. We skip the details.
\end{remark}

Putting Theorem~\ref{T:slovak} and Proposition~\ref{P:Slovak entropy} together we obtain a negative answer to the question~(2) in the Introduction:

\begin{corollary}\label{C:Tpositive FT zero}
There exists a \ds\ given by a homeomorphism of a compact space with finite positive or even infinite entropy,  whose functional envelope on the group of homeomorphisms has entropy zero.
\end{corollary}

\section{Surjections of Slovak spaces}\label{lastsection}

In this section we are going to study surjective maps of the Slovak spaces constructed in the preceding section, and we are going to show that every non-invertible surjection has a fixed point. In particular, if $\varphi\colon\widetilde F\to\widetilde F$ is minimal then it is invertible.
\smallskip

From now on, if $A$ is a subset of $X$, we will write $\widetilde A$ to denote $\pi^{-1}(A)$.

\begin{lemma}\label{composant lemma}If $\varphi:\widetilde F\to\widetilde F$ is a continuous surjection then
$\varphi^{-1}(\widetilde\gamma)=\widetilde\gamma$ (hence, by surjectivity, $\varphi(\widetilde\gamma)=\widetilde\gamma$).
\end{lemma}
\begin{proof} Note that $\varphi(\widetilde\gamma)$ is not the whole space (it has too few
path-components for that). So, by Lemma \ref{compan}, $\varphi(\widetilde\gamma)$ is contained in one composant. It now suffices to show that the complement of $\widetilde\gamma$ is mapped into itself. The path-components of that complement are composants $\widetilde\alpha$ of $\widetilde F$ different from $\widetilde\gamma$. Each of them is both pathwise connected and dense in $\widetilde F$ (see (C1)) and so must be its image $\varphi(\widetilde\alpha)$. Thus this image is contained in a dense path-component of $\widetilde F$. The path-components of $\widetilde\gamma$ are not dense, hence $\varphi(\widetilde\alpha)$ is contained in a composant different from $\widetilde\gamma$, which ends
the proof. \hfill $\square$
\end{proof}

Let $C_n$ denote the path-component of $\widetilde\gamma$ containing the interval $W_n$ (see Figure~1). By Lemma \ref{composant lemma} (and continuity of $\varphi$) $\varphi(C_{n})$ is
a subset of some (unique) path-component $C_{m}$. We define $\zeta\colon\Z\to\Z$ by the rule $\zeta(n)=m \iff \varphi(C_{n})\subset C_{m}$. Note that if $\zeta(n)=m$ then $\varphi(\overline{C_{n}})=\varphi(C_{n}\cup W_{n+1})\subset \overline{C_{m}} \subset C_m\cup C_{m+1}$. This implies that either $\zeta(n+1)=\zeta(n)$ or $\zeta(n+1)=\zeta(n)+1$. Let us denote $\mathbb D=\{n\in\Z:\zeta(n+1)=\zeta(n)\}$.

\begin{lemma}\label{mini1}
If\,\ $\mathbb D$ is bounded from below or above then $\varphi=\widetilde T^l$ for some $l\in\Z$.
\end{lemma}

\begin{proof}
Assume $\mathbb D$ is bounded from above (the other case is symmetric). Therefore there exist $l, n_0\in\Z$, such that $\zeta(n)=n+l$ for all $n\geq n_0$. This means that for $n> n_0$,
$\varphi$ sends the intervals $W_{n}$ into the intervals $W_{n+l} = \widetilde T^l(W_n)$. Hence, on the dense set $\bigcup_{n>n_0}W_n$, we have the equality $\pi\circ\varphi = T^l\circ\pi$. This equality (between two continuous maps from $\widetilde F$ into $X$) holds on a closed set, hence it holds on the whole $\widetilde F$. In particular, it holds on the $\widetilde T$-invariant union of the singleton fibers by $\pi$. Here, the equality can be rewritten as $\varphi = \pi^{-1}\circ T^l \circ \pi$. But on this set $\pi^{-1}\circ T^l \circ \pi$ equals $\widetilde T^l$. We have shown that the equality $\varphi = \widetilde T^l$ holds on a dense set, which clearly concludes the proof. \hfill $\square$
\end{proof}

Note that it follows from the proof of the lemma \ref{mini1} that for any continuous and surjective $\varphi\colon\widetilde F\to\widetilde F$ the set $\mathbb D$ is either empty or unbounded from both sides.

\begin{lemma}\label{mini2}
If a function $\varphi\colon\widetilde F\to\widetilde F$ is continuous and surjective then $\varphi$ is a minimal homeomorphism or it has fixed point.
\end{lemma}

\begin{proof}
By lemma \ref{mini1} it is enough to consider functions $\varphi$ such that the set $\mathbb D$ is unbounded from both sides. Therefore it is easy to see that there exists $n\in\Z$ such that $\zeta(n)=n$. That means $\varphi(\overline{C_{n}})\subset\overline{C_{n}}$. It is well known that the (closed) topologist's sine curve has the fixed point property (this is in fact true for every arc-like continuum, see \cite[Corollary 12.30]{Na}). Therefore $\varphi$ has a fixed point in $\overline{C_{n}}$. \hfill $\square$
\end{proof}

\begin{corollary}
There exists a non-degenerate continuum of type $(1,0)$ which is not homeomorphic to the circle.
\end{corollary}
Note that in fact in theorem \ref{T:slovak} we have constructed an uncountable family of different continua which are all of type $(1,0)$.

\begin{remark}
An analogous construction applied to an irrational rotation of the circle (see Remark \ref{circ}) does not produce a space of type $(1,0)$. There exist minimal homeomorphisms of the circle which preserve the backward orbit (with respect to the rotation) of the selected point $x_0$ and whose forward orbit of $x_0$ is disjoint from the rotation-orbit of $x_0$. Many of such homeomorphisms can be lifted and then extended to the closure of the graph of $F$, producing a non-invertible minimal map. Again, we skip more details.
\end{remark}


\end{document}